\documentclass[english, letterpaper,12pt,reqno]{amsart}

\pdfoutput=1

\usepackage[utf8]{inputenc}
\usepackage[T1]{fontenc}
\usepackage{lmodern}
\usepackage{amsthm, amssymb, amsmath, amsfonts, mathrsfs}
\usepackage[dvipsnames,svgnames,x11names]{xcolor}
\definecolor{mPurple}{HTML}{8959a8}
\usepackage[defaultcolor=purple]{changes}
\usepackage{microtype}

\usepackage[pagebackref,colorlinks=true,pdfpagemode=none,urlcolor=blue,
linkcolor=darkblue,citecolor=darkblue]{hyperref}
\usepackage{pdfcomment}
\usepackage{amsmath,amsfonts,amssymb,amsthm}
\usepackage{amsthm, amssymb, amsmath, amsfonts, mathrsfs}

\usepackage{mathrsfs}
\usepackage{MnSymbol}
\usepackage{scalerel} 

\usepackage{color}
\usepackage{accents}

\usepackage{bbm}

\definecolor{darkred}{rgb}{0.9,0.1,0.1}
\definecolor{darkgreen}{rgb}{0,0.5,0}
\definecolor{darkblue}{rgb}{0, 0.23, 0.46}

\usepackage{amsmath}
\usepackage{amssymb}
\usepackage{amsfonts}
\usepackage{mathrsfs}
\usepackage{mathtools}

\usepackage[shortlabels]{enumitem}
\usepackage{amscd} 

\usepackage[margin=1in,footskip=.25in]{geometry}

\newtheorem{theorem}{Theorem}[]
\newtheorem{lemma}[theorem]{Lemma}

\newtheorem{proposition}[theorem]{Proposition}

\theoremstyle{remark}
\newtheorem{remark}[theorem]{Remark}

\renewenvironment{proof}[1][Proof]{ {\itshape \noindent {#1.}} }{$\Box$
\medskip}

\theoremstyle{plain}

\newcommand{\E}{\mathbb{E}}

 \usepackage{bm}

\def\bbE{\mathbb E}

\def\tr{{\rm tr}}

\begin{document}

\title{Intermittency of geometric Brownian motion on \textbf{SL}(\lowercase{n}) }

\author[Sefika Kuzgun]{Sefika Kuzgun}
\author[Felix Otto]{Felix Otto}
\author[Christian Wagner]{Christian Wagner}

\address[Sefika Kuzgun]{Max Planck Institute for Mathematics in the Sciences,\newline
Inselstraße 22, 04103 Leipzig, Germany.}

\address[Felix Otto]{Max Planck Institute for Mathematics in the Sciences,\newline
Inselstraße 22, 04103 Leipzig, Germany.}

\address[Christian Wagner]{%
	Institute of Science and Technology Austria,\newline
	Am Campus 1,
	3400 Klosterneuburg, Austria.
}

\begin{abstract}

This short note is motivated by a recently discovered connection 
between a drift-diffusion process in $n$-dimensional Euclidean space
with a divergence-free drift sampled from 
a stationary and isotropic Gaussian ensemble of critical scaling on the one hand,
and a geometric Brownian motion on $\textbf{SL}(n)$ on the other hand.
This can be seen as a tensorial form of a stochastic exponential;
it thus is naturally intermittent, which transfers to the pair distance
of the drift-diffusion process.

\smallskip

In this note, we quantify the intermittency of the geometric Brownian motion $\{F_\tau\}_{\tau\ge0}$ 
on $\textbf{SL}(n)$ also in dimensions $n>2$. We do so in two (related) ways: 
1) by identifying the  
exponential growth rate for the $2p$-th stochastic moment $\mathbb{E}|F_\tau|^{2p}$
with its anomalous dependence on $p$ (and $n$), and 2) by quantifying a non-tightness of
$|F_\tau|^2/\mathbb{E}|F_\tau|^2$ as $\tau\uparrow\infty$. 
It is the second property that transmits to the drift-diffusion process.

\smallskip

The arguments rely on stochastic analysis: We write $\{F_\tau\}_{\tau\geq 0}$
as the solution of $dF=F_\tau\circ dB$ with $\{B_\tau\}_{\tau\geq 0}$ a Brownian motion
on the Lie algebra $\mathfrak{sl}(n)$. The arguments leverage isotropy:
The diffusion projects onto the spectrum of the Gram matrix $G=F^*F$,
as captured by ${\rm tr}G^p$. 

\end{abstract}
\maketitle


\section{Introduction}
\noindent
We are interested in the  process $\{F_\tau\}_{\tau \geq 0}$ that solves the Stratonovich stochastic differential equation (SDE)
\begin{align}\label{F}
    dF=F_\tau\circ dB, \quad F_{\tau=0}=id ,
\end{align}
where $\{B_\tau\}_{\tau \geq 0}$ is a Brownian motion on ${\frak{sl}}(n):=\{\tr B=0\}$ that satisfies the following three assumptions:
\begin{enumerate}[label=(B\arabic*)]
    \item \label{brownian1} $OB O^{-1}=_{\text{law}}B$ for all $O\in \textbf{O}(n):=\{O^*O=id\}$,
    \item \label{brownian2} $\bbE[B_{\tau} B_\tau] =0$, and
    \item\label{brownian3} $\bbE[B_\tau^* B_\tau]=\tau \text{id}$. 
    \end{enumerate}
Since the chain rule holds for Stratonovich SDEs, $F\in \textbf{SL}(n):=\{\text{det} F=1\}$. In \cite{MOW25}, it is proved that for $n\geq 2$, there is a unique Brownian motion on $\frak{sl}(n)$ that satisfies the Assumption \ref{brownian1}-\ref{brownian3}. We are interested in exploring the intermittent behavior of $F$ in \eqref{F}. Our first result is on the moments:
\begin{proposition}[Intermittency] \label{prop:intermittency}For any $p\geq 1$ integer, we have 
\begin{align*}
   n \exp\left((p+\frac{2p(p-1)}{n+2})\tau \right) \leq  \E |F_\tau|^{2p}\leq n^{p} \exp\left((p+\frac{2p(p-1)}{n+2})\tau \right),
\end{align*}
where $|F|^2:=\tr F^*F$ denotes the square of the Frobenius norm.
\end{proposition}
\noindent
Proposition~\ref{prop:intermittency} implies that 
\begin{align}
    \E|F_\tau|^{2p} \sim_{n,p}\footnotemark \left(\E^p |F_\tau|^2\right)^{1+\frac{2(p-1)}{n+2}}. \label{2pthmoment}
\end{align}
\footnotetext{Here $A \sim_{n,p} B$ means that there exist constants $c=c(n,p)$ and $C=C(n,p)$ depending only on $n$ and $p$ such that $cB\leq A\leq CB$.} 

\medskip \noindent
The fact that $2p$-th moment scales with a rate much larger than the $p$-th power of the second moment amounts to strongly non-Gaussian and intermittent behavior. Proposition~\ref{prop:intermittency} generalizes \cite[Lemma 2]{MOW25} for $n=2$ to general $n\ge 2$.

\medskip \noindent
The interest in studying the geometric Brownian motion comes from its
intriguing relation to the drift-diffusion process 
\begin{align}
    dX=b(X_t)dt+\sqrt{2}dW \label{dd}
\end{align}
with $b$ being a divergence-free and time-independent vector field. One is interested in the case where $b$ is sampled from a stationary and isotropic Gaussian ensemble with the scaling 
 \begin{align}
   \label{b} b(\mu \cdot)=\frac{1}{\mu}b \text{ in law for all } \mu>0,
\end{align}
in which convection and diffusion balance at every scale. After implementing an ultraviolet cutoff, w.l.o.g.\ at scale 1, $b$ is fixed up to a single constant which we describe in terms of $\mathbb{E}|b|^2=\epsilon^2 \frac{ n }{ 4 } $, where $\epsilon \ll 1$ can be interpreted as the P\'{e}clet number. It is now well-known that the mean-square displacement displays super-diffusive behavior $\frac{1}{2t}\bbE|X_t|^2 \approx \lambda(t)$ for $t \gg 1$, where 
\begin{align*}
\lambda(t):=\sqrt{1+\frac{\epsilon^2}{2} \ln(1+t)}.
\end{align*}
This behavior was rigorously established with increasing precision in \cite{TV12, CHT22, CMOW22,ABK24, MOW25}. 

\medskip \noindent
The relationship between the geometric Brownian motion process $F$ in \eqref{F} and the drift-diffusion process in \eqref{dd} was discovered in \cite{MOW25} on the level of the expected position, $u(x,t)$, (w.r.t. $W_t$) of the process $\{X_t\}_{t\geq 0}$ starting from $X_{t=0}=x$. More precisely, one obtains $\nabla u(0,t) \approx F_{\tau(T)}$ in law on average over $t\in (0,T)$, where $$\tau(T):=\ln \lambda(T), $$ see \cite[Theorem 2 \& (19)]{MOW25}.

\medskip \noindent
Our result shows that the intermittency of the drift-diffusion equation depends less on topology (stream lines of the divergence-free $b$ are closed iff $n=2$) and more on geometry ($\textbf{SL}(n)$ has less curvature as $n$ increases). 

\medskip \noindent
While Proposition~\ref{prop:intermittency} clearly expresses intermittency at the level of $F$, we are not able to transfer this higher-moment information to $\nabla u$. This is because we can capture the proximity of $\nabla u$ and $F$ only on the second moments. Hence we need to capture the shadow of the intermittency on the level of the second moment. In fact, it comes as a non-tightness result in the second moment in the sense that the extreme tails contain a substantial fraction of the second moment.
 \begin{proposition}[Non-tightness] \label{prop:nontightness}
    For $\tau \gg_n 1$\footnote{Here $\tau \gg_n 1$ means that there is a possibly large constant $C(n)$ depending only on $n$ such that the statement is true for any $\tau$ with $\tau \geq C(n)$.}, we have %
	\begin{align}\label{prop02}
		\E | F_{ \tau } |^2 I \Big( { \textstyle \frac{ 1 }{ n } } | F_{ \tau } |^2 \leq { \textstyle \frac{ 1 }{ e  } } \big( { \textstyle \E \frac{ 1 }{ n } } | F_{ \tau } |^2 \big)^{ \frac{ n + 4 }{ n + 2 } } \Big)
		\leq \Big( \frac{ 1 }{ 2 } + \frac{ C(n) }{ \sqrt{\tau} } \Big) \E| F_{ \tau }  |^2  .
	\end{align}
 where $I(A)$ is the indicator function of the event $A$.	
\end{proposition}
\begin{remark}\label{cor:nontightness} The information in Proposition~\ref{prop:nontightness} can easily be upgraded to the lower bound as is done in \cite[Theorem 3 \& (118)]{MOW25}
 \begin{align*}
\E  |F_\tau|^{2p} \gtrsim_{n,p}  \bbE^{1+\frac{n+4}{n+2}(p-1)}  |F_\tau|^{2},
\end{align*}
which is weaker than \eqref{2pthmoment} but not completely unrelated. This exponent is precisely the linearization of the quadratic exponent $p+\frac{2p(p-1)}{n+2}$ around $p=1$. 
\end{remark}

\noindent
The results in Propositions~\ref{prop:intermittency} and ~\ref{prop:nontightness} were previously proven in the case $n=2$ (see \cite[Lemma 2, Lemma 4]{MOW25}). The argument relied on the fact that the law of $R=\frac{1}{2}\tr F^*F$ can be identified as the Itô evolution %
\begin{align*}
    dR= R_{\tau} d\tau +\sqrt{R_\tau^2-1} dw, \quad R_{\tau =0}=1,
\end{align*} 
where $\{w_\tau\}_{\tau\geq 0}$ is a one-dimensional Brownian motion. In any dimension $n$, as a consequence of the $\textbf{O}(n)$-invariance, it is not difficult to see that one obtains an evolution for the $(n-1)$ quantities $\tr F^*F, \tr (F^*F)^2,\dots ,\tr (F^*F)^{n-1}$. However, we are unable to leverage this evolution in order to characterize the evolution of $R$. In this paper, we take a different route: Relying on the observation that the quantities $\E \tr (F^*F)^2 $ and $\E\tr^2 F^*F $ satisfy a linear system of ODEs which is closed in any dimension $n$ and can be solved explicitly, we learn that $\bbE\tr (F^*F)^2$ can be well approximated by $\bbE \tr^2 F^* F$. This in turn allows us to approximately close the equation for $ R $ in dimensions $ n > 2 $.

\medskip
 \noindent
Note that $|F_\tau|^2$ is close to a stochastic exponential, which is consistent with Proposition~\ref{prop:intermittency} and \ref{prop:nontightness}, namely $\ln |F_\tau|^2$ behaves like a Gaussian random variable with mean $(1-\frac{2}{n+2})\tau$ and variance $\frac{4\tau}{n+2}$.

\subsection*{Acknowledgment}
We thank Anna Wienhard and Corentin Le Bars for helpful discussions concerning random walks on groups. We thank Peter Morfe and Ofer Zeitouni for helpful discussions. S.K. and F.O. acknowledge funding by the Deutsche Forschungsgemeinschaft (DFG, German Research Foundation) - CRC/TRR 388 "Rough Analysis, Stochastic Dynamics and Related Fields" - Project ID 516748464.

\section{Strategy of Proof}

\noindent
Assumptions \ref{brownian1}-\ref{brownian3} on the Brownian motion $B$ can be equivalently expressed in terms of the geometric Brownian motion $F$ as follows: 
\begin{lemma}[\cite{MOW25}, Lemma 1]
    Let $\{F_\tau\}_{\tau \ge 0}$ be the solution to SDE \eqref{F} where $\{B_\tau\}_{\tau \ge 0}$ satisfies Assumptions \ref{brownian1}-\ref{brownian3}. Then  
    \begin{enumerate}[label=(F\arabic*)]
    \item Isotropy: \label{F1} $OF O^{-1}=_{\text{law}}F$ for all $O\in \textbf{O}(n)$,
    \item\label{F2} Itô=Stratonovich: $F_\tau \circ dB=F_\tau dB$,
    \item \label{F3} Normalization: $\bbE[|F_\tau|^2 ]=n e^\tau $.
    \end{enumerate}
\end{lemma}
\noindent
Observe that \ref{F3} is nothing but Proposition~\ref{prop:intermittency} for $p=1$ in which case the upper and lower bounds coincide. 

\noindent
It turns out to be convenient to work with the Gram matrix
\begin{align}
    \label{:G} G:=F^*F \quad \text{ so that } \quad |F|^2=\tr G.
\end{align}
Note that $G$ is symmetric and positive definite; its square root appears in the polar factorization of $F$. Hence passing from $F$ to $G$ amounts to work with the quotient space $\textbf{SL}(n)/\textbf{SO}(n)$ of the Lie group $\textbf{SL}(n)$ by its maximal compact subgroup $\textbf{SO}(n)$ as advocated by differential geometers. Since in particular $G$ is diagonalizable with nonnegative eigenvalues, we have the inequality %
\begin{align}\label{convexity}
    \tr G^p \leq \tr ^pG \quad \text{for } p\geq 1.
\end{align}
It turns out to be convenient to monitor $\tr G^p$ next to $\tr^pG=|F|^{2p}$. The Stratonovich evolution for $G$, $G^2$ are easily derived:
\begin{align}
    dG&=G_\tau \circ dB+\circ dB^* G_\tau, \label{G} \\
    dG^2&=G_\tau\circ dB G_\tau+G^2_\tau \circ dB+\circ dB^* G^2_\tau+G_\tau\circ dB^* G_\tau.\label{G2}
\end{align}
Taking the trace in \eqref{G} and \eqref{G2} and using the cyclic property $\tr GB=\tr BG$ for any matrices $G$ and $B$, makes the symmetric part of $B$ appear
 \begin{align}
    d\tr G &=\tr G_\tau (\circ dB+\circ dB^*) =2\tr G_\tau \circ dB_{sym}, \label{trgstr} \\
    d\tr G^2&=2\tr G^2_\tau (\circ dB+\circ dB^*) =4\tr G^2_\tau\circ dB_{sym}. \label{trG2str} 
\end{align}
Using the chain rule for Stratonovich integral, we also have \begin{align}
    d\tr^2 G&= 4\tr G_\tau \tr G_\tau\circ dB_{sym}.\label{tr2Gstr}
\end{align}
\noindent
The following lemma characterizes $B_{sym}$.
\begin{lemma}[\cite{MOW25},(74)]\label{lem:b}
   Let $\{B_\tau\}_{\tau\geq 0}$ be a Brownian motion on $\frak{sl}(n)$ satisfying Assumptions \ref{brownian1}-\ref{brownian3}. Then the symmetric part $B_{sym}:=\frac{B+B^*}{2}$ and skew-symmetric part $B_{skew}:=\frac{B-B^*}{2}$ are independent Brownian motions with the property that 
   \begin{align}
    \bbE[B_{sym,\tau}^2]= -\bbE[B_{skew,\tau}^2]=\frac{\tau id}{2}.\label{Bsymvar}
\end{align}
\end{lemma}
\noindent
In order to pass from the Stratonovich to the Itô formulation, we need to compute Itô-Stratonovich correction terms and quadratic variations. Let us start with \eqref{trgstr}, which using \eqref{G} has the correction:\footnote{ For brevity and ease of notation, we purposefully avoid using $d[\tr \int_0^{\cdot} G_{\tau} dB  \, B_{sym}]$ in favor of $d[\tr G B B_{sym}]$. This is justified by the following computational rule for Itô integrals: Given a linear form $ f $, we can define the quadratic variation $ [ f ( \cdot \, B B_{ sym } ) ] $ as the linear map $ G \mapsto [ f ( G \, BB_{ sym } ) ] $. Equipped with this definition we have
\begin{align*}
	\Big[ f \big( \int_0^{ \cdot } G_{ \tau } d B \, B_{ sym } \big) \Big] 
	= \int_0^{ \cdot } G_{ \tau } .  d [ f ( \cdot \, B B_{ sym } ) ] .
\end{align*}
In words: to compute the variation $ [\tr \int_0^{\cdot} G dB  \, B_{sym}] $, we replace $ G $ by a deterministic objects, and integrate the result against the process $ G $.}
 \begin{align} \begin{aligned}
    d[\tr G B_{sym}]&=d[\tr G (B B_{sym} +B_{sym}B^*)]\\& =  d[\tr G (B_{ sym}^2+B_{skew}B_{sym}+B_{sym}^2- B_{sym}B_{skew})]\\ & = 2 d[\tr G B^2_{ sym}]\overset{\eqref{Bsymvar}}{=}\tr G_\tau d\tau. \end{aligned} \label{itostra}
\end{align}
Hence the Itô formulation of the equation \eqref{trgstr} becomes \begin{align}
    d\tr G &=2\tr G_\tau dB_{sym}+\tr G_\tau d\tau. \label{trg} 
\end{align} 
The quadratic variation of the martingale part is $  4  d[\tr  G B_{ sym} \tr  G B_{ sym}]$. The following lemma will be used to characterize this quadratic variation as well as the variations coming from equations \eqref{trG2str}, \eqref{tr2Gstr}. The nontrivial dependence on dimension $n$ arises from this lemma the proof of which relies on combinatorial arguments. 
\begin{lemma} \label{lem:quadraticvar} For any symmetric matrices $ G,\tilde G$, we have 
\begin{align}
      d\left[\tr G B_{sym} \tr \tilde G B_{sym}\right]&=  \frac{d\tau}{\alpha_n} \left(n \tr G \tilde G - \tr G \tr \tilde G \right),\label{trGBtrGB}\\  d\left[ \tr GB_{sym}  \tilde G B_{sym}\right]&=\frac{d\tau}{2\alpha_n}\left( (n-2) \tr G \tilde G +n\tr G \tr \tilde G \right), \label{trGBGB}
\end{align}
where 
\begin{align}
    \alpha_n := (n-1)(n+2). \label{alpha}
\end{align}
\end{lemma}
\noindent
Using Lemma~\ref{lem:quadraticvar}, the quadratic variation of the martingale part of the equation \eqref{trg} becomes
\begin{align}\label{qv}
    d[\tr  G B_{ sym} \tr  G B_{ sym}]=\frac{1}{\alpha_n} (n \tr  G_\tau^2-\tr^2 G_\tau )d\tau.
\end{align} 
Moreover, the Itô form of the equation \eqref{trG2str} becomes
\begin{align}\label{trg2}
d\tr G^2=4\tr G_\tau^2dB_{sym}+\frac{2n}{\alpha_n} \tr^2 G_\tau d\tau +\left(2+\frac{2(n-2)}{\alpha_n}\right)\tr  G_\tau^2d\tau. 
\end{align}
More details on the computation of the correction is given after \eqref{TrGpst-ito}. 
Note that this is not closed because of the presence of $\tr^2 G$. Similarly, using the Itô formula, one sees \begin{align}
\begin{aligned}    d\tr^2 G &=4 \tr G_\tau \tr G_\tau dB_{\text{sym}} +2 \tr^2 G_\tau d\tau+ \frac{4}{\alpha_n} (n \tr  G_\tau^2-\tr^2 G_\tau )d\tau\\&=4 \tr G_\tau \tr G_\tau dB_{\text{sym}} +\left(2-\frac{4}{\alpha_n}\right) \tr^2 G_\tau d\tau+ \frac{4n}{\alpha_n} \tr  G_\tau^2 d\tau. \end{aligned}\label{tr2g}
\end{align}
Doing similar computations on integer powers leads to the following differential equations on the level of expectations:
\begin{lemma}\label{lem:odes}
   Let $\{F_\tau\}_{\tau\geq 0}$ be the process satisfying \eqref{F} and recall that $G=F^*F$. Then we have
   \begin{align}\label{trpg}
 \frac{ d\bbE\tr^p G}{d\tau}& =(p-\frac{2p(p-1)}{\alpha_n}) \bbE\tr^p G_\tau +\frac{2p(p-1)n}{\alpha_n}  \bbE\tr^{p-2} G_\tau\tr  G_\tau^2  
\end{align}
for any $p>0$ and 
\begin{align}\label{trgp}
   \frac{ d \bbE\tr G^p}{d\tau}=(p+\frac{p(p-1)(n-2)}{\alpha_n}) \bbE\tr G_\tau^p  +\frac{pn}{\alpha_n}\sum_{j=1}^{p-1}\bbE\tr G_\tau^j \tr G_\tau^{p-j}  
\end{align}
for $p\geq 1$ integer. 
\end{lemma}
\noindent
With these equations at hand, we are practically done by using inequality \eqref{convexity} or variants of it to get 
\begin{align}\label{lpestimates}
n   e^{(p+\frac{2p(p-1)}{n+2})\tau} \leq  \bbE\tr  G_\tau^p \leq   \bbE\tr^p  G_\tau \leq  n^pe^{(p+\frac{2p(p-1)}{n+2})\tau},
\end{align}
which transfers to $F$, yielding Proposition~\ref{prop:intermittency}.
\noindent
Now let us discuss the proof of Proposition~\ref{prop:nontightness}. The observation to implement the proof established in \cite{OW24} in case $n>2$ is delicate. A priori, the Itô evolutions satisfied by observables of $\tr G$, say $\zeta(\tau,R)$ for $R=\tr G$ are not closed similarly to \eqref{trg2}. However only in $n=2$, this system reduces to a single evolution of $R=\tr G$ because $\tr G^2=4(\tr^2 G - \det G)$, and only then were we able to deduce intermittency using $R$, see \cite[Section 3]{MOW25}. Instead, we take a different road: It is based on the observation that the expectations, $\bbE \tr^2 G_{\tau}$, $\bbE \tr G_\tau^2$ satisfy a closed linear system of ODEs which can be solved explicitly. 
\begin{lemma}\label{Lem:2by2}
  It holds
   \begin{align*}
    \frac{d}{d\tau}\begin{bmatrix}
        \bbE \tr^2 G \\ \bbE \tr G^2
    \end{bmatrix}  = A \begin{bmatrix}
        \bbE \tr^2 G_\tau \\ \bbE \tr G_\tau^2
    \end{bmatrix}, \quad \quad
   \begin{bmatrix}
        \bbE \tr^2 G_{\tau=0} \\ \bbE \tr  G_{\tau=0}^2
    \end{bmatrix} = \begin{bmatrix}
    n^2\\n
\end{bmatrix}
\end{align*}
where 
\begin{align*}
    A=(2-\frac{4}{\alpha_n}) \begin{bmatrix}
        1&0\\
        0&1
    \end{bmatrix}+\frac{2n}{\alpha_n}\begin{bmatrix}
        0&2\\
        1&1
    \end{bmatrix}.
\end{align*}
The matrix $A$ has the eigenvalues $\lambda_1=2+\frac{4}{n+2}> \lambda_2=2-\frac{2}{n-1}$ with left eigenvectors $v_1=\begin{bmatrix}
    1\\2
\end{bmatrix}$ and $v_2=\begin{bmatrix}
    1\\-1
\end{bmatrix}$. As a consequence, for all $\tau>0$, 
\begin{align}\label{difference}
 \bbE  \tr^2G_{\tau}-\bbE \tr G_\tau^2  = n(n-1) e^{\lambda_2\tau}.  
\end{align}
\end{lemma}
\medskip\noindent
The identity \eqref{difference} reveals that the difference $\bbE \tr^2 G_\tau - \bbE \tr G^2_\tau$ grows at a slower exponential rate than either $\bbE \tr^2 G_\tau$ or $\bbE \tr G^2_\tau$ both of which grows like $e^{\lambda_1\tau}$.  This observation enables us to write an approximate Itô evolution for 
\begin{align}\label{R}
    R_\tau:=\tr G_\tau.
\end{align}
Indeed, recall from \eqref{trg} and \eqref{qv} that $R$ satisfies  \begin{align}
dR=2\tr G_\tau dB_{sym}+R_\tau d\tau, \quad \quad R_{\tau=0}=n ,\label{RSDE}
\end{align}
    with quadratic variation \begin{align*}
        d[RR]&=\frac{4}{\alpha_n}(n\tr G_\tau^2-R^2_\tau)d\tau .
    \end{align*}
   By definition of $\alpha_n$, see \eqref{alpha}, this
   can be rewritten as
\begin{align}\label{QuadVarR}
        d[RR]&=\frac{4}{n+2}R_\tau^2 d\tau+\frac{4n}{\alpha_n} (\tr G_\tau^2-\tr^2 G_\tau)d\tau .   \end{align}
        By the above lemma it is legitimate to neglect the last term in \eqref{QuadVarR} to obtain the approximate SDE
      \begin{align}\label{Rsde}
    dR \approx \frac{2}{\sqrt{n+2}} R_\tau dw+R_{\tau}d\tau, \quad \quad R_{\tau=0}=n,
\end{align}
where $w$ is a one-dimensional Brownian motion. Hence, $R$ is approximately a geometric Brownian motion $ ne^{(1-\frac{2}{n+2})\tau +\frac{2}{\sqrt{n+2}} w_\tau}$, which is consistent with the observation that $\ln R=\ln |F|^2$ is close to being normal with mean $(1-\frac{2}{n+2})\tau$ and variance $\frac{4\tau}{n+2}$ as claimed in the introduction. This motivates us to consider the following change of variables
\begin{align}\label{Rhat}
  R=n e^{\tau} \hat R ,
\end{align}
under which \eqref{Rsde} turns into
 \begin{align}\label{Rhatapprox}
   d\hat R \approx \frac{2}{\sqrt{n+2}}\hat R_{ \tau} dw, \quad \quad \hat R_{ \tau=0}=1. 
\end{align}
\medskip

\noindent
To capture the above approximation more quantitatively, we consider an observable $\zeta=\zeta( \tau, r)$ solving the Kolmogorov backward PDE \eqref{Rsde} for $R$  (when treating \eqref{Rsde} as an exact equality)
\begin{align}\label{backwardpde}
	   \frac{\partial \zeta}{\partial \tau}+r \frac{\partial \zeta}{\partial r}+\frac{2 }{n+2}r^2\frac{\partial^2\zeta}{\partial  r^2}=0.
\end{align}
Using the change of variables \eqref{Rhat} on the level of the variables $ r $, $ \hat r $ (i.e.~$ r = n e^{ \tau } \hat r $) the PDE \eqref{backwardpde} transforms into\footnote{We will detail the upcoming computations in the proof of Lemma \ref{lem:zeta}.}
\begin{align}\label{backwardpde2}
    \frac{\partial \zeta}{\partial \tau}+\frac{2}{n+2}\hat r^2\frac{\partial^2\zeta}{\partial \hat r^2}=0,
\end{align}
which is the Kolmogorov equation for \eqref{Rhatapprox}. 
The form of the observable in Proposition~\ref{prop:nontightness} also motivates us to make the ansatz that
 \begin{align}
    \zeta(\tau,\hat r)=\hat r \hat 
\zeta(\tau,\hat r) \label{zetatohatzeta}
\end{align} 
with $\hat \zeta$ that solves 
\begin{align}\label{zetahat}
\frac{\partial \hat \zeta}{\partial \tau}+2\hat r\frac{\partial \hat \zeta}{\partial\hat r}+\frac{2}{n+2}\hat r^2 \frac{\partial^2 \hat \zeta}{\partial \hat r^2}=0.
\end{align} 
Finally, due to the structure of \eqref{zetahat} it is natural to introduce the variable
\begin{align}\label{sigmahat}
    \hat \sigma = \ln \hat r ,
\end{align}
w.r.t.~which \eqref{zetahat} takes form \eqref{zetahat2} below.

\begin{lemma}\label{lem:zeta} 
   Suppose $\hat \zeta=\hat\zeta(\tau,\hat\sigma)$ solves \begin{align}\label{zetahat2}
\frac{\partial \hat \zeta}{\partial \tau}+\frac{2}{n+2}\frac{\partial \hat \zeta}{\partial\hat \sigma }+ \frac{2}{n+2}\frac{\partial^2 \hat \zeta}{\partial \hat \sigma^2}=0,
\end{align}
then \begin{align} \label{zetahatineq}
    \bbE \hat R \hat \zeta(\tau, \hat R) - \hat \zeta(\tau=0,\hat r= 1 ) \lesssim_{ n } \int_0^{\tau} d\tau' e^{-(1-\frac{\lambda_2}{2})\tau'} \sup_{\hat \sigma} \left|\frac{\partial \hat \zeta}{\partial\hat \sigma}+ \frac{\partial^2 \hat \zeta}{\partial\hat \sigma^2} \right|  ,
 \end{align}
 where $\lambda_2$ is defined in Lemma~\ref{Lem:2by2}. 
\end{lemma}
\noindent
Note that $\lambda_2<2$ is important to leverage \eqref{zetahatineq} in the proof of Proposition \ref{prop:nontightness}.

\section{Proofs}
\subsection{Proofs of Lemmas}

\begin{proof}[Proof of Lemma~\ref{lem:quadraticvar}]
Recall that the space of symmetric bilinear forms on $\{G:G^*=G\}$ that are invariant under conjugation by $ \textbf{O}(n)$ is two dimensional and spanned by $\tr G \tilde G$ and $\tr G \tr \tilde G$, see for example \cite[Lemma 16]{MOW25} for a proof. Note that 
$$(G, \tilde G) \mapsto   \bbE \tr GB_{sym,\tau} \tr \tilde GB_{sym,\tau}$$
is such a symmetric bilinear form thanks to the $\textbf{O}(n)$ invariance in law of $B$ in \ref{brownian1}, which transmits to $B_{sym}$. Since $B_{sym}$ is a Brownian motion, its tensorial quadratic covariation agrees with its second moment, which grow linearly in $\tau$; hence for non-random $G,\tilde G$, 
$$[\tr GB_{sym} \tr \tilde GB_{sym}]_{\tau}=\bbE \tr GB_{sym,\tau} \tr \tilde GB_{sym,\tau}=\text{const}\, \tau.$$ 
As a consequence, we have
\begin{align}\label{qvgeneral}
       d[ \tr GB_{sym} \tr \tilde GB_{sym}] =(a_{n} \tr G \tilde G+b_{n} \tr G \tr\tilde G ) d\tau
    \end{align}
for some constants $a_n,b_n$ only depending on $n$. 
\noindent
In order to find $a_n,b_n$, we start with the choice $G=\tilde G=id$. On the one hand, we have $d[\tr B_{sym}\tr B_{sym}  ]=0$ and on the other hand $n a_{n}$ and $n^2 b_n$ which this leads to 
\begin{align}
      0=n a_n+n^2 b_n.\label{system1}
    \end{align}
Second, we test with $G=\tilde G=e_i\otimes e^j +e_j\otimes e^i$ (without implicit sum), where $e^i=e_i^*$ denotes the dual basis. On one hand, 
\begin{align}\label{BG}
    B_{sym}G&=e_i \otimes B_{sym}e^j+e_j \otimes B_{sym}e^i 
    \end{align}
and by the cyclic property of the trace and the symmetry of $B_{sym}$, we have
     \begin{align*}
  \tr GB_{sym}&=  \tr B_{sym}G   =B_{sym}e^j.  e_i+B_{sym}e^i. e_j=2 B_{sym} e^i. e_j 
\end{align*}
which leads to
\begin{align*}
        d[\tr GB_{sym}\tr \tilde G B_{sym}]=4d[(B_{sym}e^i. e_j) (B_{sym}e^i. e_j)]. 
    \end{align*}
   Summing $i,j$ over $\{1,\dots,n\}$, the left-hand side of the equation \eqref{qvgeneral} becomes $4d[\tr B_{sym}^2 ]$ which is equal to $2n d\tau$ by 
   \begin{align}\label{qvtrb2}
        [\tr B^2_{sym}]_\tau=\bbE \tr B_{sym,\tau}^2=\frac{n\tau}{2}.
        \end{align}
  On the other hand, we have
  \begin{align*}
      \tr G&= e^j. e_i+e^i . e_j=2\delta_i^j, \quad \quad  \text{ and thus}, & \tr^2 G&=4 \delta_i^j,\\
      G^2&=\delta_i^j( e_i\otimes e^j+e_j\otimes e^i)+e_i\otimes e^i+e_j\otimes e^j, \text{ and thus}, & \tr G^2&=2\delta_i^j +2. 
  \end{align*}
  Again summing over $i,j$ for $\tr^2G $ and $\tr G^2$, and then plugging in \eqref{qvgeneral} we get that
  \begin{align}
      2n=2n(n+1)a_n +4nb_n .\label{system2}
  \end{align} 
The system \eqref{system1}, \eqref{system2} has the solution 
\begin{align*}
        a_{n}=\frac{n}{\alpha_n} \quad \text{and} \quad b_{n}=-\frac{1}{\alpha_n}
    \end{align*}
    where $\alpha_n$ is as defined in \eqref{alpha}. Hence, finally using \eqref{Bsymvar} leads to \eqref{trGBtrGB} in Lemma~\ref{Lem:2by2}.

\medskip \noindent
Similarly, the bilinear form 
\begin{align*}
    (G,\tilde G) \mapsto \bbE \tr GB_{sym,\tau}\tilde G B_{sym,\tau}
\end{align*}
is $\textbf{O}(n)$ invariant by $\textbf{O}(n)$ invariance in law of $B$ in \ref{brownian1}. Moreover, it is symmetric thanks to the cyclic property of the trace.
So, by the same argument as above, we can write 
\begin{align*}
  d[ \tr GB_{sym} \tilde GB_{sym}] =(a_{n} \tr G \tilde G+b_{n} \tr G\tr\tilde G) d\tau
\end{align*}
for some constants $a_n$, and $b_n$ depending only on $n$. In order to find them, we will test on the same pair of matrices. This in turn implies that the right-hand sides of the equations stay the same. Let us compute the left-hand sides: For $G=\tilde G=id$, we have $d[\tr GB_{sym}\tilde G B_{sym}] =\frac{n}{2}d\tau$ by \eqref{qvtrb2} and for $G=\tilde G =e_i\otimes e^j +e_j\otimes e^i$, using \eqref{BG} we have 
\begin{align*}
    B_{sym}GB_{sym}G&=B_{sym} e^j . e_i e_i\otimes B_{sym}e^j+ B_{sym} e^i . e_j e_j\otimes B_{sym}e^i\\ & +B_{sym} e^j . e_j e_i\otimes B_{sym}e^i+B_{sym} e^i . e_i e_j\otimes B_{sym}e^j
\end{align*}
and taking the trace and using the symmetry of $B_{sym}$, we get
\begin{align*}
     \tr B_{sym}GB_{sym}\tilde G= \tr GB_{sym}\tilde G B_{sym} =2B_{sym}e^j. e_iB_{sym} e^j. e_i+2B_{sym}e^i. e_iB_{sym} e^j. e_j.
\end{align*}
Finally summing over $i,j$, the left-hand side of the equation becomes 
\begin{align*}
    \sum_{i,j=1}^n d[\tr G B_{sym}G B_{sym}]=2d[\tr B_{sym}^2]+\underbrace{2d[\tr B_{sym}\tr B_{sym}}_{=0}]=2d[\tr B_{sym}^2]=nd\tau
\end{align*}
where we have used \eqref{qvtrb2} again. Hence we obtain the following system for $a_n$ and $b_n$
\begin{align*}
 \frac{n}{2}&=  na_{n}+n^2 b_{n},\\
n&=2n(n+1)a_{n}+4n b_{n},
    \end{align*}
 which has the solution 
 \begin{align*}
     a_{n}=\frac{n-2}{2\alpha_n}, \quad \text{and} \quad b_{n}=\frac{n}{2\alpha_n}
 \end{align*}   
 and we obtain the second claim \eqref{trGBGB} in Lemma~\ref{Lem:2by2}.
\hfill\end{proof}

\begin{proof}[Proof of Lemma~\ref{lem:odes}]
    Note that $G_\tau$ and $G^2_\tau$ satisfy \eqref{G}, \eqref{G2} and more generally for $p\geq 1$ integer, $G_\tau^p$ satisfies 
    \begin{align}
    dG^p&=\sum_{j=1}^p G_\tau^{j-1}dGG_\tau^{p-j}=\sum_{j=1}^p G_\tau^{j}\circ dB G_\tau^{p-j}+\sum_{j=1}^p G_\tau^{j-1} \circ dB^* G_\tau^{p-j+1} . \label{Gp}
\end{align}
Taking the trace in \eqref{Gp}, and using cyclic property of the trace, we get the following Stratonovich SDE: 
\begin{align*}
    d \tr G^p=\tr  dG^p=p\tr G_\tau^p(\circ dB+\circ dB^*)=2p\tr G_\tau^p\circ dB_{{sym}}.
\end{align*}
Now, we want to switch to the Itô formulation which can be done using the Itô-Stratonovich correction term as follows: 
\begin{align}\label{TrGpst-ito}
    d \tr G^p=2p \tr G_\tau^p\circ dB_{{sym}}=2p\tr  G_\tau^p dB_{{sym}}+p d[\tr G^pB_{{sym}}]
\end{align}
where using \eqref{Gp} we see 
\begin{align*}\begin{aligned}
    d[\tr  G^pB_{sym}]&=d[\tr \left(\sum_{j=1}^p G^j B G^{p-j}+G^{j-1} B^*G^{p-j+1}\right)B_{sym}]\\ &=d[\tr (G^p B B_{ sym}+B^*G^p B_{sym}) +\sum_{j=1}^{p-1} \tr G^j (B+B^*)G^{p-j}B_{ sym}].
  \end{aligned}
\end{align*}
Together with the cyclic property of the trace, we rewrite it
as 
\begin{align}\begin{aligned}
    d[\tr  G^pB_{sym}]=d[\tr G^p( BB_{sym}+ B_{sym}B^*) +2\sum_{j=1}^{p-1} \tr G^j B_{sym}G^{p-j}B_{ sym}].\end{aligned} \label{qvgp}
\end{align}
 We can also replace $B$, $B^*$ in the first two terms using $B_{sym},B_{skew}$ and thanks to their independence, see Lemma~\ref{lem:b}, we obtain
 \begin{align}
\begin{aligned}
    &d[\tr G^p (BB_{ sym}+B_{sym}B^*)]\\&= d[\tr G^p (B_{ sym}^2+B_{skew}B_{sym}+B_{sym}^2- B_{sym}B_{skew})]\\
  & = 2 d[\tr G^p B^2_{ sym}]\overset{\eqref{Bsymvar}}{=}\tr G_\tau^p d\tau. \end{aligned}   \label{indep}
\end{align}
Inserting \eqref{indep} and \eqref{trGBGB} in \eqref{qvgp}, we finally get
\begin{align}\begin{aligned}
    d[\tr  G^pB_{sym}]
    &=\tr  G_\tau^p d\tau+\frac{1}{\alpha_n}\sum_{j=1}^{p-1}\left( (n-2) \tr  G_\tau^p+n \tr  G_\tau^j \tr  G_\tau^{p-j}\right)d\tau \\
    &= (1+\frac{(p-1)(n-2)}{\alpha_n}) \tr  G_\tau^p d\tau +\frac{n}{\alpha_n}\sum_{j=1}^{p-1}\tr  G_\tau^j \tr  G_\tau^{p-j} d\tau.\end{aligned} \label{qvgp2}
\end{align}
Plugging \eqref{qvgp2} into \eqref{TrGpst-ito}, we get the Itô-SDE for $\tr G^p$:
 \begin{align}\label{tracegp}
    d \tr  G^p=2p\tr  G_\tau^p dB_{sym}+(p+\frac{p(p-1)(n-2)}{\alpha_n}) \tr  G_\tau^p d\tau +\frac{pn}{\alpha_n}\sum_{j=1}^{p-1}\tr  G_\tau^j \tr  G_\tau^{p-j} d\tau.
\end{align}
Using the Itô formula with the function $f(x)=x^p$ in \eqref{trg} and the quadratic variation \eqref{qv}, we have 
 \begin{align}\begin{aligned}
 d\tr^p  G
 =2p \tr^{p-1}\; G_\tau \tr \; G_\tau dB_{\text{sym}} +\left( p-\frac{2p(p-1)}{\alpha_n}\right) \tr^p\; G_\tau d\tau +\frac{2np(p-1)}{\alpha_n}  \tr^{p-2}\;G_\tau\tr\; G_\tau^2  d\tau.    
 \end{aligned}\label{tracepg} 
\end{align}
Finally, taking the expectations in \eqref{tracepg} and \eqref{tracegp} concludes the proof. 
\hfill \end{proof}

\begin{proof}[Proof of Lemma~\ref{lem:zeta}] We start from \eqref{RSDE} and \eqref{QuadVarR} that imply for $ \zeta = \zeta ( \tau, R_{ \tau } ) $
\begin{align}\label{eqn02} \nonumber
	d \zeta 
	&= \text{martingale} + \Big( \frac{ \partial \zeta }{ \partial \tau } +R_{ \tau }  \frac{ \partial \zeta }{ \partial r } \Big) d \tau 
	+ \frac{ 1 }{ 2 } \frac{ \partial^2 \zeta }{ \partial r^2 } d [ R \, R ] \\
	&= \text{martingale} + \Big( \frac{ \partial \zeta }{ \partial \tau } + R_{ \tau } \frac{ \partial \zeta }{ \partial r } + \frac{ 2 }{ n + 2 } R_{ \tau }^2 \frac{ \partial^2 \zeta }{ \partial r^2 }  \Big) d \tau
	+ \frac{ 2 n }{ \alpha_{ n } } \Big( \tr G_{ \tau }^2 - \tr^2 G_{ \tau } \Big) \frac{ \partial^2 \zeta }{ \partial r^2 } d \tau .
\end{align}
Due to the term $\tr  G^2$, this differential equation \eqref{eqn02} even on the level of expectation is not closed. To overcome this, we will invoke Lemma~\ref{Lem:2by2}, more specifically identity \eqref{difference}. 

\noindent
Choosing our observable $ \zeta $ in \eqref{eqn02} according to \eqref{backwardpde} as we argued to be equivalent to \eqref{zetahat2} the sum in the first parenthesis vanishes. Indeed, recall from \eqref{Rhat}, \eqref{zetatohatzeta} and \eqref{sigmahat} that we change variables according to
\begin{align*}
	r = n e^{ \tau } \hat r ,
	\quad \zeta = \hat r \, \hat \zeta ,
	\quad \hat \sigma = \ln \hat r ,
\end{align*}
so that $ \frac{ \partial }{ \partial r } = e^{ - \tau } \frac{ 1 }{ n } \frac{ \partial }{ \partial \hat r }  $, $ \frac{ \partial }{ \partial \hat r } = \frac{ 1 }{ \hat r } \frac{ \partial }{ \partial \hat\sigma } $ and readily
\begin{align}\label{Trafo}
	\frac{ \partial \zeta }{ \partial \tau } = \hat r \frac{ \partial \hat \zeta }{ \partial \tau } - \hat r \Big( \hat \zeta + \frac{ \partial \hat \zeta }{ \partial \hat \sigma } \Big)  ,
	\quad r \frac{ \partial \zeta }{ \partial r } = \hat r \Big( 1 + \frac{ \partial  \hat \zeta }{ \partial \hat\sigma } \Big) , 
	\quad  r^2 \frac{ \partial^2 \zeta }{ \partial r ^2} = \hat r \Big( \frac{ \partial \hat \zeta }{ \partial \hat\sigma } + \frac{ \partial^2 \hat \zeta }{ \partial \hat\sigma^2 } \Big) .
\end{align}
Thus, for the term of our interest
\begin{align*}
	\frac{ \partial \zeta }{ \partial \tau } + r \frac{ \partial \zeta }{ \partial r }  + \frac{ 2 }{ n + 2 } r^2\frac{ \partial^2 \zeta }{ \partial r^2 }
	\stackrel{(\ref{Trafo})}{ = } \hat r \Big( \frac{ \partial \hat \zeta }{ \partial \tau } + \frac{ 2 }{ n + 2 } \Big( \frac{ \partial \hat \zeta }{ \partial \hat\sigma } + \frac{ \partial^2 \hat\zeta }{ \partial \hat\sigma^2 } \Big) \Big) ,
\end{align*}
which makes the operator in \eqref{zetahat2} appear, and \eqref{eqn02} implies the preliminary
\begin{align}\label{eqn01}
	\frac{d  \E \zeta }{ d \tau }
	\lesssim_n \sup_{ \hat \sigma } \Big| r \frac{ \partial^2 \zeta }{ \partial r^2 } \, \Big|
	\, \E \frac{ 1 }{ R_\tau } \big( \tr^2 G_{ \tau } - \tr G_{ \tau }^2 \big)  
	\stackrel{ (\ref{Trafo}) }{ \approx_n } \sup_{ \hat \sigma } e^{ - \tau } \Big| \frac{ \partial \hat \zeta }{ \partial \hat\sigma } + \frac{ \partial^2 \hat \zeta }{ \partial \hat\sigma^2 }
	 \Big|
	\, \E \frac{ 1 }{ R_{ \tau } } \big( \tr^2 G_{ \tau } - \tr G_{ \tau }^2 \big)  .
\end{align}
On the remaining expectation we use the elementary inequality
\begin{align*}
	0 \leq \frac{ 1 }{ R_{ \tau } } \big( \tr^2 G_{ \tau } - \tr G_{ \tau }^2 \big)
	\stackrel{ (\ref{R}) }{ = } \frac{ 1 }{   \tr G_{ \tau } } \big( \tr^2 G_{ \tau } - \tr G_{ \tau }^2 \big) 
	\leq \sqrt{ \tr^2 G_{ \tau } - \tr G_{ \tau }^2 } ,
\end{align*}
so that by Jensen's inequality and Lemma \ref{Lem:2by2}
\begin{align*}
	\E \frac{ 1 }{ R_{ \tau } } \big( \tr^2 G_{ \tau } - \tr G^2_{ \tau } \big)
	\leq \sqrt{ \E ( \tr^2 G_{ \tau } - \tr G_{ \tau }^2 ) } 
	\stackrel{ (\ref{difference}) }{ \lesssim_n } e^{ \frac{ \lambda_2 }{ 2 } \tau } .
\end{align*}
Hence \eqref{eqn01} turns into:
\begin{align*}
	\frac{d  \E \hat R \hat \zeta }{ d \tau } 
	\lesssim_n e^{ - ( 1 - \frac{ \lambda_2 }{ 2 } ) \tau } \sup_{ \hat \sigma } \Big| \frac{ \partial \hat \zeta }{ \partial \hat\sigma } + \frac{ \partial^2 \hat \zeta }{ \partial \hat\sigma^2 }
	\Big| .
\end{align*}
Integrating the estimate and using the second item in \eqref{Rhatapprox} yields the claim in \eqref{zetahatineq}.\hfill\end{proof}

\subsection{Proof of Proposition~\ref{prop:intermittency}}
Recall from Lemma~\ref{lem:odes}, we have for $p\geq 1$ integer
\begin{align*}
  \frac{  d\bbE[\tr  G^p]}{d\tau} =(p+\frac{p(p-1)(n-2)}{\alpha_n})\bbE \tr  G_\tau^p  +\frac{pn}{\alpha_n}\sum_{j=1}^{p-1}\bbE \tr  G_\tau^j \tr  G_\tau^{p-j} 
\end{align*}
 with initial condition $\tr G_{\tau=0}=\tr id=n$. Using $\tr  G^j \tr  G^{p-j} \geq \tr  G^p$ for $0\leq j\leq p$, we can bound the second term 
 \begin{align*}
    \sum_{j=1}^{p-1}\bbE \tr  G_\tau^j \tr  G_\tau^{p-j} \geq (p-1) \bbE \tr G_\tau^p
\end{align*}
which leads to the differential inequality 
\begin{align*}
      \frac{  d\bbE\tr  G^p}{d\tau} \geq (p+\frac{p(p-1)(2n-2)}{\alpha_n}) \bbE\tr  G_\tau^p
\end{align*}
where $\alpha_n$ is as defined \eqref{alpha}. Hence we get 
\begin{align*}
   \bbE\tr  G_\tau^p \geq n e^{(p+\frac{2p(p-1)}{n+2})\tau} .
\end{align*}
On the other hand, recall again from Lemma~\ref{lem:odes}, $\bbE \tr^p G_\tau$ satisfies  
\begin{align*}
 \frac{ d\bbE\tr^p G}{d\tau}& =(p-\frac{2p(p-1)}{\alpha_n}) \bbE\tr^p G_\tau +\frac{2p(p-1)n}{\alpha_n}  \bbE\tr^{p-2} G_\tau\tr  G_\tau^2 
\end{align*}
with initial condition $\bbE\tr^p G_{\tau=0}=n^p$. This time using $\tr  G^2\leq \tr^2  G$, we get the differential inequality
\begin{align*}
   \frac{d\bbE \tr^p  G}{d\tau}& \leq (p+\frac{p(p-1)(2n-2)}{\alpha_n}) \bbE\tr^p  G_\tau 
\end{align*}
which leads to
\begin{align*}
   \bbE\tr^p  G_\tau \leq n^pe^{(p+\frac{2p(p-1)}{n+2})\tau} .
\end{align*}
Hence we obtain \eqref{lpestimates} which is a restatement of Proposition~\ref{prop:intermittency}.
\qed

\subsection{Proof of Proposition~\ref{prop:nontightness}} 

For parameters $ \tau^*  \geq 0 $ and $ \hat{\sigma}^* $ that we will choose later we consider the terminal condition
\begin{align}\label{tc}
    \hat{\zeta}( \tau^*,\hat{\sigma})\begin{cases}
        = 1, \text{ if } \hat{\sigma} \leq \hat{\sigma}^*,\\
        \in [0,1] \text{ for all } \hat{\sigma},\\
        =0, \text{ if }\hat{\sigma}\geq \hat{\sigma}^*+1,
    \end{cases}
 \text{ with } \left| \frac{\partial^2 \hat{\zeta}( \tau^* ,\hat{\sigma}) }{\partial \hat{\sigma}^2}\right| \lesssim 1
 \end{align}
 for the equation \eqref{zetahat2} in Lemma~\ref{lem:zeta}.
It is readily seen from \eqref{tc} that we have
 \begin{align}
\bbE \hat R_{ \tau^* } \hat{\zeta}(\tau^*, \hat R_{ \tau^* }) \stackrel{ (\ref{tc}) }{ \geq } \bbE \hat R_{ \tau^* } I( \hat R_{ \tau^* } \leq \hat{R}^*) 
\quad \text{provided} \quad \hat\sigma^* =: \ln \hat R^{ * }. \label{ononehand}
\end{align}
Our goal is to choose $ \hat\sigma^* $ large enough so that at initial time $ \tau = 0 $ it holds
\begin{align}\hat{\zeta}( \tau=0, \hat\sigma = 0 ) \leq \frac{1}{2} , \label{initial}
\end{align}
so that Lemma~\ref{lem:zeta} implies
\begin{align}
 \bbE \hat R_{ \tau^* } I( \hat R_{ \tau^* } \leq \hat{R}^*) - \frac{ 1 }{ 2 }
 \stackrel{ (\ref{ononehand}) }{ \leq } \bbE   \hat{R}_{ \tau^* } \hat{\zeta}( \tau^*, \hat R_{ \tau^* }) - \frac{ 1 }{ 2 }
 \stackrel{ \eqref{zetahatineq} \& (\ref{initial}) }{ \lesssim_{ n } } \int_0^{\tau^*}d \tau\,  e^{ -  ( 1 - \frac{ \lambda_2 }{ 2 }) \tau}  \sup_{\hat \sigma}\left|\frac{\partial \hat \zeta}{\partial\hat \sigma}+ \frac{\partial^2 \hat \zeta}{\partial\hat \sigma^2} \right|  . \label{zetahatineqb}
\end{align}

\medskip
\noindent
To obtain \eqref{initial} let us note that the heat kernel for \eqref{zetahat2} with terminal data at $ \tau^*$ is a Gaussian with mean $- \frac{ 2 }{ n + 2 } ( \tau^*- \tau) $ and variance $ \frac{ 4 }{ n + 2 } ( \tau^* - \tau ) $. Moreover, by \eqref{tc} we know $\hat{\zeta}( \tau^*,\hat{\sigma}) \leq I(\hat{\sigma}\leq \hat{\sigma}^*+1)$ so that by the comparison principle we have
\begin{align*}
    \hat{\zeta}(\tau,\hat{\sigma}) \leq \Phi \Big( \frac{-\hat{\sigma}+\hat{\sigma}^*+1 - \frac{ 2 }{ n + 2 } ( \tau^*- \tau)}{\sqrt{ \frac{ 4 }{ n + 2 } ( \tau^*-\tau ) } } \Big)
    \stackrel{\text{at}\, \tau = 0,  \hat{\sigma} = 0 }{ = }\Phi \Big(\frac{\hat{\sigma}^*+1- \frac{ 2 }{ n + 2 } \tau^* }{\sqrt{ \frac{ 4 }{ n + 2 } \tau^*}} \Big) ,
\end{align*}
where $\Phi$ is the cumulative distribution function of a standard normal random variable. This turns into \eqref{initial} provided we choose $ \hat{\sigma}^* \leq \frac{ 2 }{ n + 2 } \tau^*- 1 $, or equivalently,
\begin{align}\label{eqnHatRStar}
	\hat{R}^*
	\stackrel{ (\ref{ononehand}) }{ = } e^{ \hat \sigma^* }
	\leq e^{ \frac{ 2 }{ n + 2 } \tau^*-1} .
\end{align}

\medskip

%
%
\noindent
To conclude the proof, it is left to show that the integral in \eqref{zetahatineqb} is small. To this end, note that using the estimates for the terminal data 
\begin{align*}
    \sup_{\hat\sigma } \left|\frac{\partial \hat{\zeta}( \tau^*,\hat{\sigma})}{\partial {\hat\sigma}}\right|+\left|\frac{\partial^2 \hat{\zeta}(\tau^*,\hat{\sigma})}{\partial^2 {\hat \sigma } }\right| 
    \lesssim 1
    \quad \text{and} \quad
        \int d{\hat{\sigma}}\left|\frac{\partial \hat{\zeta}( \tau^*,\hat{\sigma})}{\partial {\hat\sigma}}\right| 
    \lesssim 1
\end{align*}
we learn (by convolution with the heat kernel) that 
\begin{align*}
 \sup_{ \hat \sigma } \left|\frac{\partial \hat{\zeta} ( \tau^* , \hat\sigma ) }{\partial \hat{\sigma}}\right|  \lesssim_{ n }  \frac{1}{ \sqrt{1+ \tau^*- \tau}}, \hskip 10pt \sup_{ \hat\sigma } \left|\frac{\partial^2 \hat{\zeta}  ( \tau^* , \hat\sigma ) }{\partial \hat{\sigma}^2}\right| 
    \lesssim_{ n } \frac{1}{ 1+ \tau^*- \tau} .
\end{align*}
Hence, the integral term in \eqref{zetahatineqb} is bounded by
\begin{align*} \int_0^{\tau*}d\tau e^{ - ( 1 - \frac{ \lambda_2 }{ 2 } ) \tau}  \sup_{\hat \sigma}\left|\frac{\partial \hat \zeta}{\partial\hat \sigma}+ \frac{\partial^2 \hat \zeta}{\partial\hat \sigma^2} \right|  \lesssim_{ n }  \int_0^{\tau^*} d\tau e^{-\frac{4}{n-1} \tau  } \frac{1}{\sqrt{1+\tau^*-\tau}} \lesssim_{ n }  \frac{ 1 }{\sqrt{\tau^*}},
\end{align*}
and with some constant $ C = C ( n ) $ estimate \eqref{zetahatineqb} turns into
\begin{align}\label{ontheotherhand}
 \bbE \hat R_{ \tau^* } I( \hat R_{ \tau^* } \leq \hat{R}^*) \leq \frac{ 1 }{ 2 } +   \frac{ C(n) }{ \sqrt{ \tau^* } }.
\end{align}
Finally, undoing the change of variables we have
\begin{align*}
	\hat R_{ \tau^* }
	\stackrel{ \eqref{Rhat} }{ = } e^{ - \tau^* } { \textstyle \frac{ 1 }{ n } } R_{ \tau^* }
	\stackrel{ \ref{F3} , (\ref{R}) \& (\ref{:G}) }{ = } \frac{ | F_{ \tau^* } |^2 }{ \E | F_{ \tau^* } |^2 } ,
	\quad 
	\hat R^{ * }
	\stackrel{ (\ref{eqnHatRStar}) }{ \leq } e^{ \frac{ 2 }{ n + 2 } \tau^* - 1 } 
	\stackrel{\ref{F3}  }{ = } \frac{ 1 }{ e } \Big( \frac{ \E | F_{ \tau^* } |^2  }{ n } \Big)^{ \frac{ 2 }{ n + 2 } } ,
\end{align*}
so that since $ \frac{ 2 }{ n + 2 } + 1 = \frac{ n +  4 }{ n + 2 } $ 
\begin{align*}
	\hat R_{ \tau^* } I( \hat R_{ \tau^* } \leq \hat{R}^*)
	= \frac{ | F_{ \tau^* } |^2 }{ \E | F_{ \tau^* } |^2 } I \Big( { \textstyle \frac{ 1 }{ n } } | F_{ \tau^* } |^2 \leq { \textstyle \frac{ 1 }{ e } } \big( { \textstyle \frac{ \E | F_{ \tau^* } |^2  }{ n } } \big)^{ \frac{ n + 4 }{ n + 2 } } \Big) .
\end{align*}
Thus \eqref{ontheotherhand} turns into \eqref{prop02}. \qed


\begin{thebibliography}{99}

\bibitem{ABK24}
	S. Armstrong, A. Bou-Rabee, and T. Kuusi.
	\newblock{Superdiffusive central limit theorem for a Brownian particle in a critically-correlated incompressible random drift.}
	\newblock{\emph{arXiv preprint} arXiv:2404.01115 (2024).}

\bibitem{CHT22}
    G. Cannizzaro, L. Haunschmid-Sibitz, and F. Toninelli.
    \newblock $ \sqrt{ \log t } $-superdiffusivity for a
Brownian particle in the curl of the 2D GFF.
    \newblock \emph{The Annals of Probability}, 50(6):2475–2498, 2022.

\bibitem{CMOW22}
    G. Chatzigeorgiou, P. Morfe, F. Otto, and L. Wang.
    \newblock The Gaussian free-field as
a stream function: asymptotics of effective diffusivity in infra-red cut-off.
    \newblock to appear in \emph{Annals of
Probability, 2025.}

\bibitem{MOW25}
    P. Morfe, F. Otto, and C. Wagner.
    \newblock A critical drift-diffusion equation: Intermittent
behavior via geometric Brownian motion on SL(n).
    \newblock \emph{arXiv preprint} arXiv:2511.15473 (2025).

\bibitem{OW24}
    F. Otto and C. Wagner.
    \newblock A critical drift-diffusion equation: intermittent behavior. \newblock \emph{arXiv
preprint} arXiv:2404.13641 (2024).

\bibitem{TV12}
    B. Tóth and B. Valkó.
    \newblock Superdiffusive bounds on self-repellent Brownian polymers and
diffusion in the curl of the Gaussian free field in $ d = 2 $.
    \newblock \emph{Journal of Statistical Physics, 147:113–131,
2012}.

\end{thebibliography}
\end{document}